\documentclass[11pt,a4paper]{amsart}
\usepackage{amssymb,amsmath,layout}
\usepackage{amssymb,amsmath,epsfig,graphics}

\theoremstyle{plain}
\newtheorem{thrm}{Theorem}[section]

\newtheorem{rmrk}[thrm]{Remark}
\newtheorem{dfn}[thrm]{Definition}

\begin{document}
% begin top matter
% ***************** macroes needed for this paper ************************

\newcommand{\SL}{\mathcal L^{1,p}( D)}
\newcommand{\Lp}{L^p( Dega)}
\newcommand{\CO}{C^\infty_0( \Omega)}
\newcommand{\Rn}{\mathbb R^n}
\newcommand{\Rm}{\mathbb R^m}
\newcommand{\R}{\mathbb R}
\newcommand{\Om}{\Omega}
\newcommand{\Hn}{\mathbb H^n}
\newcommand{\aB}{\alpha B}
\newcommand{\eps}{\ve}
\newcommand{\BVX}{BV_X(\Omega)}
\newcommand{\p}{\partial}
\newcommand{\IO}{\int_\Omega}
\newcommand{\bG}{\boldsymbol{G}}
\newcommand{\bg}{\mathfrak g}
\newcommand{\bz}{\mathfrak z}
\newcommand{\bv}{\mathfrak v}
\newcommand{\Bux}{\mbox{Box}}
\newcommand{\e}{\ve}
\newcommand{\X}{\mathcal X}
\newcommand{\Y}{\mathcal Y}
\newcommand{\W}{\mathcal W}

\numberwithin{equation}{section}

\newcommand{\RN} {\mathbb{R}^N}
\newcommand{\Sob}{S^{1,p}(\Omega)}
\newcommand{\Dxk}{\frac{\partial}{\partial x_k}}
\newcommand{\Co}{C^\infty_0(\Omega)}
\newcommand{\Je}{J_\ve}
\newcommand{\beq}{\begin{equation}}
\newcommand{\bea}[1]{\begin{array}{#1} }
\newcommand{\eeq}{ \end{equation}}
\newcommand{\ea}{ \end{array}}
\newcommand{\eh}{\ve h}
\newcommand{\Dxi}{\frac{\partial}{\partial x_{i}}}
\newcommand{\Dyi}{\frac{\partial}{\partial y_{i}}}
\newcommand{\Dt}{\frac{\partial}{\partial t}}
\newcommand{\aBa}{(\alpha+1)B}
\newcommand{\GF}{\psi^{1+\frac{1}{2\alpha}}}
\newcommand{\GS}{\psi^{\frac12}}
\newcommand{\HFF}{\frac{\psi}{\rho}}
\newcommand{\HSS}{\frac{\psi}{\rho}}
\newcommand{\HFS}{\rho\psi^{\frac12-\frac{1}{2\alpha}}}
\newcommand{\HSF}{\frac{\psi^{\frac32+\frac{1}{2\alpha}}}{\rho}}
\newcommand{\AF}{\rho}
\newcommand{\AR}{\rho{\psi}^{\frac{1}{2}+\frac{1}{2\alpha}}}
\newcommand{\PF}{\alpha\frac{\psi}{|x|}}
\newcommand{\PS}{\alpha\frac{\psi}{\rho}}
\newcommand{\ds}{\displaystyle}
\newcommand{\Zt}{{\mathcal Z}^{t}}
\newcommand{\XPSI}{2\alpha\psi \begin{pmatrix} \frac{x}{|x|^2}\\ 0 \end{pmatrix} - 2\alpha\frac{{\psi}^2}{\rho^2}\begin{pmatrix} x \\ (\alpha +1)|x|^{-\alpha}y \end{pmatrix}}
\newcommand{\Z}{ \begin{pmatrix} x \\ (\alpha + 1)|x|^{-\alpha}y \end{pmatrix} }
\newcommand{\ZZ}{ \begin{pmatrix} xx^{t} & (\alpha + 1)|x|^{-\alpha}x y^{t}\\
     (\alpha + 1)|x|^{-\alpha}x^{t} y &   (\alpha + 1)^2  |x|^{-2\alpha}yy^{t}\end{pmatrix}}
\newcommand{\norm}[1]{\lVert#1 \rVert}
\newcommand{\ve}{\varepsilon}

\title[Modica type gradient estimates, etc.]{Modica type gradient estimates for an inhomogeneous variant of the normalized $p$-Laplacian evolution}

\dedicatory{Dedicated to Enzo Mitidieri, on the occasion of his 60th birthday}

\author{Agnid Banerjee}
\address{Department of Mathematics\\University of California, Irvine \\
CA- 92697} \email[Agnid Banerjee]{agnidban@gmail.com}
\thanks{First author was supported in part by the second author's
NSF Grant DMS-1001317 and by a postdoctoral grant of the Institute Mittag-Leffler}

\author{Nicola Garofalo}
\address{Dipartimento di Ingegneria Civile, Edile e Ambientale (DICEA) \\ Universit\`a di Padova\\ 35131 Padova, ITALY}
\email[Nicola Garofalo]{rembdrandt54@gmail.com}
\thanks{Second author was supported in part by NSF Grant DMS-1001317 and by a grant of the University of Padova, ``Progetti d'Ateneo 2013"}

%
% 
% AMS information
%

\maketitle
% end top matter

\begin{abstract}
 In this paper, we  study  an inhomogeneous  variant of the normalized $p$-Laplacian evolution which has been recently  treated in \cite{BG1}, \cite{Do}, \cite{MPR} and \cite{Ju}.  We show that if the initial  datum satisfies the pointwise  gradient estimate \eqref{e:main1} a.e., then the unique solution to the Cauchy problem \eqref{main5}  satisfies  the same   gradient  estimate a.e. for all later times, see \eqref{e:main} below. A general pointwise  gradient bound for the entire bounded solutions of the elliptic counterpart of equation \eqref{main5} was first obtained in \cite{CGS}. Such estimate generalizes one obtained by L. Modica for the Laplacian, and it has connections to a famous conjecture of De Giorgi.
\end{abstract}

\section{Introduction}

Recently, there has been increasing attention about the equation of the so-called normalized $p$-Laplacian evolution
\begin{equation}\label{e:p}
|Du|^{2-p} \text{div}(|Du|^{p-2}Du)= u_t, \ \ \ \ \ \ \ \ \ 1<p<\infty,
\end{equation}
see \cite{BG1}, \cite{Do}, \cite{MPR}, \cite{Ju}, \cite{BG2} and \cite{JK}. The equation \eqref{e:p} is an evolution associated with the $p$-Laplacian that interpolates between the motion by mean curvature, which corresponds to the case $p=1$, and the heat equation, corresponding to $p=2$. In the interesting paper \cite{MPR} solutions to \eqref{e:p} have been characterized by asymptotic mean value properties. These properties are connected with the analysis of tug-of-war games with noise in which the number of rounds is bounded. The value functions for these games approximate a solution to the PDE \eqref{e:p} when the parameter that controls the size of possible steps go to zero.  The equation \eqref{e:p} also arises in image processing, see \cite{Do}, in which the Cauchy-Neumann problem was studied. In \cite{BG1} we constructed viscosity solutions to \eqref{e:p} and derived properties such as comparison principles for solutions of \eqref{e:p}, convergence of solutions as $p\to 1$, and the large-time behavior of solutions to a Cauchy-Dirichlet problem for \eqref{e:p}. We also proved unweighted energy monotonicity and a generalized Struwe's monotonicity formula. In the paper \cite{Ju} Juutinen studied the large-time behavior for $p> 2$ of solutions of \eqref{e:p}.
The case $p=\infty$ of the normalized $\infty$-Laplacian evolution was studied  in \cite{JK}. The equation  \eqref{e:p} has the advantage of being $1$-homogeneous but it has the serious disadvantage of having a non-divergence structure.

In the present paper for a  given  $T > 0$ we consider the  following Cauchy problem in $\Rn \times [0, T]$ 
\begin{equation}\label{main5}
\begin{cases}
 |Du|^{2-p}\left\{\operatorname{div}(|Du|^{p-2}Du) - F'(u)\right\} = u_t,  
\\
u(\cdot, 0)=g.
\end{cases}
\end{equation}
We suppose that $F\in C^{2,\beta}_{loc}(\R)$ for some $\beta > 0$  and $F\ge 0$.
Throughout this paper we assume $ 1 < p  \leq 2 $. We observe that, because of its non-divergence structure, when $F\not\equiv 0$ the equation \eqref{main5} does not make sense for $p >2$. As a consequence, in the case $p>2$ it presently remains an interesting open question what is the right evolution for which results similar to those in this paper can be established.

The equation in \eqref{main5} can  be considered as the parabolic counterpart of 
\begin{equation}\label{mainell}
 \text{div} (|Du|^{p-2}Du)= F'(u),
\end{equation}
which is a special case of the class of equations $\operatorname{div}(\Phi'(|Du|^2) Du) = F'(u)$ treated in \cite{CGS}. As a consequence of the results in \cite{CGS}, it follows that entire bounded (weak) solutions to \eqref{mainell} satisfy the following pointwise gradient estimate
\begin{equation}\label{est60}
|Du|^{p} \leq \frac{p}{p-1} F(u).
\end{equation}
We recall that in the linear case $p=2$ the estimate \eqref{est60} was first proved by L. Modica in  \cite{Mo}.
The estimate \eqref{est60} (in fact, a generalization of it) was employed in \cite{CGS} to provide a partial answer to a famous conjecture of De Giorgi (also known as the $\ve$-version of the Bernstein theorem for minimal graphs) asserting that entire solutions to
\begin{equation}
\Delta u = u^3 - u,
\end{equation}
such that $|u|\le 1$ and $\frac{\p u}{\p x_n} >0$, must be one-dimensional, i.e., must 
have level sets which are hyperplanes, at least in dimension $n\le 8$. In \cite{CGS} the estimate \eqref{est60} was also used to establish a result on the propagation of the zeros of a solution to \eqref{mainell}. 
We recall that the conjecture of De Giorgi has been fully solved for $n=2$ in \cite{GG} and $n=3$ in \cite{AC}, and it is known to fail for $n\ge 9$, see \cite{dPKW}. For $4 \le n \le 8$ it is still an open question. Additional fundamental progress on De Giorgi's conjecture is contained in the papers \cite{GG2}, \cite{Sa}. 

In this paper, we study the parabolic analogue of the Modica type gradient estimate \eqref{est60}.  Before stating our main results, we introduce the relevant class of solutions for the Cauchy problem \eqref{main5}:
\begin{equation*}\label{d:h}
H_T  = \{u \in C(\Rn \times [0,T])\mid x \to u(x, t)\in C^{0,1}(\Rn), ||u||_{L^{\infty}(\Rn \times [0, T])}, ||Du||_{L^{\infty}(\Rn \times [0, T])}  < \infty\}.
\end{equation*}
The notation $C^{0,1}(\Om)$ indicates the class of Lipschitz continuous functions on a given open set $\Om\subset \Rn$. 
The following is our main result.

\begin{thrm}\label{1d}
Let $g\in C^{0,1}(\Rn)$ with $||g||_{L^{\infty}(\Rn)}, ||Dg||_{L^{\infty}(\Rn)}  < \infty$. Moreover, corresponding to $g$, we assume that $F$ satisfies the assumption \eqref{assump} below. Then, for every $T>0$ there exists a unique solution $u$ to the Cauchy problem \eqref{main5} in the class $H_T$. Furthermore, if the initial datum $g$  satisfies the following gradient estimate for a.e. $x\in \Rn$
\begin{equation}\label{e:main1}
 |Dg(x)|^{p} \leq \frac{p}{p-1} F(g(x)),
\end{equation}
 then, at any given time $t > 0$ one has for a.e. $x\in \Rn$
\begin{equation}\label{e:main}
|Du(x, t)|^{p} \leq  \frac{p}{p-1} F(u(x, t)).
\end{equation}
\end{thrm}

\begin{rmrk}\label{r}
The assumption \eqref{assump} below is used to  assert the existence of solutions  in the class $H_T$  via a regularization scheme described in the subsequent sections, see Remark \ref{assumpt1}. The hypothesis \eqref{assump} is however not needed when $1<p < 2$, see Remark \ref{assumpt2} below. In addition, such a regularization scheme  is also  crucially  employed to justify the computations in Section \ref{proofs}.  Now, when $p=2$, any solution in the  class $H_T$ is a classical solution, a fact  which follows from the  parabolic regularity theory. Hence, in this case one does not need to apply any further regularization scheme. In conclusion,  if we a priori assume that the solution $u$ belongs to the class $H_T$, then we obtain the following version of Theorem \ref{1d}.
\end{rmrk}

\begin{thrm}\label{1e}
Let $1<p\le 2$, and for some $0<T\le \infty$ let $u\in H_T$ be a solution to 
\begin{equation}
|Du|^{2-p}\left\{\operatorname{div}(|Du|^{p-2}Du) - F'(u)\right\} = u_t,
\end{equation}
where $F\in C^{2, \beta}_{loc}(\R)$ for some $\beta > 0$,  and $F\ge 0$.  If at some time level $t_{0}$  $u(\cdot,t_{0})$ satisfies \eqref{e:main}, then $u(\cdot,t)$ satisfies \eqref{e:main} for all $t_{0}\le t \le T$ ($t<\infty$ if $T=\infty$).
\end{thrm}
\begin{rmrk}
Note that  unlike the hypothesis in  Theorem \ref{1d}, in Theorem \ref{1e} we do not require that $F$ satisfy \eqref{assump}. See Remark \ref{r} above.
\end{rmrk}

Theorem \ref{1d} and Theorem \ref{1e} can be considered as a parabolic  analogue in the case $1 < p \leq 2$ of the above mentioned  result in \cite{CGS} which states that  an   entire bounded solution to \eqref{mainell} satisfies the estimate \eqref{est60} except that in our situation we are only able to  assert that the estimate  \eqref{e:main} holds a.e. in $\Rn$. It remains an open question as to whether  the solution $u$ in Theorem \ref{1d} has higher regularity so that one can assert that the estimate \eqref{e:main} holds pointwise  everywhere. In the next result we show that, under an additional assumption on the initial datum $g$, this is true when $n=2$.

\begin{thrm}\label{2d}
Let $n=2$, and let $u$, $g$ be as in Theorem \ref{1d}. Furthermore, if  the initial datum $g$ has bounded  derivatives up to order two there exists $\alpha\in (0,1)$ depending only on $p$ such the solution $u(\cdot,t) \in C^{1,\alpha}$ for every $t>0$. Consequently, the gradient estimate \eqref{e:main} holds pointwise everywhere.
\end{thrm}
 
We conclude with an application of the estimate \eqref{e:main}. The following result can be thought of as theorem on the propagation of zeros for solutions of the Cauchy problem \eqref{main5}.

\begin{thrm}\label{3d}
Suppose that the initial datum $g$ satisfies \eqref{e:main1}, and let $u$ be the solution as in Theorem \ref{1d}.   If $  F(u( x_{0}, t_{0})) = 0 $ for some point  $(x_{0}, t_{0})$, then $u(\cdot, t_{0})$ is constant.
\end{thrm}
\noindent \textbf{Acknowledgment:} The paper was finalized  during the first author's stay at the Institut Mittag-Leffler during the semester long program \emph{Homogenization and Random Phenomenon}. The first  author would like to thank the Institute and the organizers of the program for the kind hospitality and the  excellent working conditions.

\section{Preliminaries}

Suppose that $u$ be a solution to the equation \eqref{main5}. We begin by observing that, after some formal computations, we have the following equation in non-divergence form 
\begin{equation}\label{e:inhom}
\left(\delta_{ij} + (p-2)\frac{u_i u_j}{|Du|^2}\right)u_{ij}=  |Du|^{2-p} f(u) +  u_{t},  
\end{equation}
where  $f = F'$ (see \cite{BG1} for similar formal computations in the homogeneous case $F\equiv 0$). 
Following  [CGG], we  now introduce  the  following notion of viscosity solution to the equation in  \eqref{main5}.

\begin{dfn}\label{D: def1}
A   function  $u\in C(\Rn \times [0, T))\cap L^{\infty}(\Rn \times  [0, T))$ is called a \emph{viscosity subsolution} of \eqref{e:inhom},  provided that for every $\phi \in  C^{2}(\Om \times (0, T))$ such that
\begin{equation}
u - \phi\quad  \text{has a local maximum at}\quad   z_0 \in \Om \times  (0, T),
\end{equation}
then either
\begin{equation}
\begin{cases}
\phi_t  + |D\phi|^{2-p} f(u) 

 \leq \left(\delta_{ij} + (p-2)  \frac{ \phi_{i}\phi_{j}}{|D\phi|^{2}}\right)\phi_{ij}\quad\ \ \text{at}\quad z_0,
\\
\text{if}\   D\phi(z_0) \not= 0,
\end{cases}
\end{equation}
or
\begin{equation}
\begin{cases}
\underset{|a| = 1}{\inf}\ \left\{\phi_t  + |D\phi|^{2-p} f(u)- (\delta_{ij}  +  (p-2)  a_{i}a_{j})\phi_{ij}\right\}  \leq 0\ \ \ \text{at}\ z_0,
\\
\text{if}\   D\phi(z_0) = 0.
\end{cases}
\end{equation}
A function $u$ is a viscosity supersolution if $v = - u$ is a viscosity subsolution. Finally, $u$ is a viscosity solution if it is at the same time a subsolution and a supersolution.
\end{dfn} 

Similarly to the case  $F = 0$, by arguing as in Proposition 2.8 in \cite{BG1} we  have the following equivalent definition.
\begin{dfn}\label{D:def2}
A   function  $u\in C(\Rn \times [0,T))\cap L^{\infty}(\Rn \times  [0,T))$ is called a \emph{viscosity subsolution} of \eqref{e:inhom},  provided that for every $\phi \in  C^{2}(\Om \times (0, T))$ such that
\begin{equation}
u - \phi\quad  \text{has a local maximum at}\quad   z_0 \in \Om \times  (0, T),
\end{equation}
then 
\begin{equation}
\begin{cases}
\phi_t  + |D \phi|^{2-p} f(u) \leq \left(\delta_{ij} + (p-2)  \frac{ \phi_{i}\phi_{j}}{|D\phi|^{2}}\right)\phi_{ij}\quad \text{at}\quad z_0,
\\
\text{if}\   D\phi(z_0) \not= 0,
\end{cases}
\end{equation}
or
\begin{equation}
\begin{cases}
\phi_t  +|D\phi|^{2-p}f(u) \leq \left(\delta_{ij} + (p-2)  a_{i}a_{j}\right)\phi_{ij}\quad \text{at} \quad z_0,\  
\\
\text{for some}\quad  a \in  \Rn\quad  \text{with}\quad  |a| \leq 1, \text{if}\   D\phi(z_0) = 0.
\end{cases}
\end{equation}
Analogous definitions for supersolutions, and for solution.
\end{dfn}

\section{Maximum modulus principle}

In this short section we establish  a maximum modulus theorem for viscosity solutions of  \eqref{main5} which will be needed subsequently.

\begin{thrm}\label{max mod}
Let $u$ and  $v$  be two  bounded  continuous solutions in $\Rn\times [0,T]$ to \eqref{main5}  which are  globally  Lipschitz in the space variable. Let 
\begin{equation}\label{e:150}
||u||_{L^{\infty}(\Rn \times (0, T))}, ||Du||_{L^{\infty}(\Rn \times (0, T))}, ||v||_{L^{\infty}(\Rn \times (0, T))}, ||Dv||_{L^{\infty}(\Rn \times (0, T))}  \leq C.
\end{equation}
Then, there exists a constant  $M = M(C)$ such that 
\begin{equation}\label{e:151}
|| u - v||_{L^{\infty}(\Rn \times (0, T))} \leq e^{M T} || u(\cdot, 0) - v(\cdot, 0)||_{L^{\infty}(\Rn \times (0, T))}.
\end{equation}
\end{thrm}

\begin{proof}
First, we let $G\in C^2(\R)$ be a compactly supported  real-valued function such that $G(w)= F(w)$ when $|w| \leq 2 C + 1$. Let now $\phi$ be a test function such that $u- \phi$ has a local extremum at  a point  $z_0 = (x_0, t_0)$. From  \eqref{e:150}, it follows that $|D\phi| \leq C$, and a similar  conclusion is also true when $u$ is replaced by $v$. Therefore, if we  define  $Q(y)= |y|^{2-p}$ if $|y| \leq 2C$ and $Q(y)= 2^{2-p} C^{2-p}$ when  $|y| \geq 2C$,  we have  that both $u$ and $v$ are viscosity solutions to  
\begin{equation}
w_t + Q(Dw) G^{'}(w)= (\delta_{ij} + (p-2) \frac{w_i w_j}{|Dw|^2})w_{ij}.
\end{equation}
This equation obeys the hypothesis of Theorem 4.1 in \cite{GGIS}.  As a consequence, \eqref{e:151} follows  from a slight modification of the arguments in the proof of Theorem 4.1 in \cite{GGIS} which can be found for instance in Theorem 1.2.1 in \cite{Zh}. Note that the modification is similar to the one employed for the  case  $F=0$ in proof of Theorem 3.4 in \cite{BG1}.

\end{proof}

\section{Existence of solutions}\label{existence1}

In this section we establish the solvability of the Cauchy problem \eqref{main5} when the initial datum $g\in C^{0,1}(\Rn)$, i.e., $g$ is globally Lipschitz and bounded. 
With this objective in mind, for any $\ve > 0$ we consider the approximating  Cauchy problem
\begin{equation}\label{e:ap1}
\begin{cases}
u^{\ve}_{t}  +  (\ve^2 + |Du^{\ve}|^2)^{1 - p/2} f(u^{\ve}) =  a^{\ve}_{ij}(Du^{\ve}) u^{\ve}_{ij}
\\
u^\ve(\cdot,0) = g,
\end{cases}
\end{equation}
where we have let $f = F'$, and
\begin{equation}
a^{\ve}_{ij}(\sigma)=  \delta_{ij} + (p-2) \frac{\sigma_i \sigma_j}{\ve^{2} + |\sigma|^{2}},\ \ \ \ i, j = 1,...,n.
\end{equation}
It is easily seen that for every $\sigma\in \Rn$ and every $\xi\in \Rn$ the following uniform ellipticity condition is satisfied, independently of $\ve>0$,
\begin{equation}\label{ue}
\min\{1,p-1\}\ |\xi|^2 \le a^{\ve}_{ij}(\sigma)\xi_i\xi_j  \leq \max\{1,p-1\}\ |\xi|^2.
\end{equation}
Proceeding as follows we first obtain a unique bounded  classical  solution $u^\ve$ to
\eqref{e:ap1}.

We let $M = ||g||_{L^{\infty}(\Rn)}$.  In correspondence of the initial datum $g$ we assume that the nonlinearity $F$ in \eqref{main5} satisfy the following hypothesis: there exist constants $q$, $M_{1}$, $M_{2}$, all depending on $M$, such that one has
\begin{equation}\label{assump}
\begin{cases}
 - q \leq M_{1}  \leq - M\ \text{and}\  M \leq M_2 \leq q,
\\
f(M_{1}) \leq 0 \leq   f(M_{2}). 
\end{cases}
\end{equation}
We remark immediately that assumption \eqref{assump} will be needed only in the case $p=2$, but not when $1<p < 2$. We also note that for the typical representatives of nonlinearities $f(u) = u^{3}-u$,  $f(u) = \sin\ u$ in  \eqref{main5} the assumption \eqref{assump} is satisfied.

Assuming \eqref{assump} let now $\tilde{F}$ be a compactly supported, $C^{2, \beta}(\R)$ function such that  $ \tilde{F} =  F$ for $ |u| \leq 2 q + 1$. We first suppose additionally that $g$ is smooth and has bounded derivatives of all orders. We  take a sequence of smooth domains $\Omega ^{N} \nearrow  \Rn$. Given any $T>0$, we consider the finite cylinders $\Om^N_T = \Om^N \times (0,T)$, and indicate with $\p_p \Om^N_T = (\p \Om^N \times (0,T))\cup (\Om^N \times \{0\})$ its parabolic boundary. For each $N\in \mathbb N$,   and $\ve>0$, we solve the Cauchy-Dirichlet  problem  
\begin{equation}\label{goodapp}
\begin{cases}
u^{\ve,N}_{t}  +  (\ve^2 + |Du^{\ve, N}|^2)^{1 - p/2} \tilde{F}'(u^{\ve, N}) =  a^{\ve}_{ij}(Du^{\ve,N}) u^{\ve,N}_{ij},\ \ \ \ \text{in}\ \Om^N_T,
\\
u^{\ve,N} = g\ \ \ \text{on}\ \p_p \Om^N_T\ \ \ (\text{one should keep in mind that}\ g(x,t) = g(x)).
\end{cases}
\end{equation}
The existence of classical solutions  $u^{\ve, N}$, such that  $\underset{\Om^N_T}{\sup}\ ||Du^{\ve, N}|| < \infty$, is guaranteed by  Theorem $4.2$, p. 559 in \cite{LU}. Because of the   boundedness of gradient,  one can see that $u^{\ve, N}$ satisfies an equation which obeys the  hypothesis of the comparison principle, Theorem $9.1$ in \cite{Li}. Moreover, because of  \eqref{assump} $M_{1}$ is a subsolution and $M_{2}$ is a  supersolution to such an equation. Therefore, from the comparison principle Theorem \ref{max mod} above we  conclude that  $|u^{\ve, N}|$  is bounded from above by $q$, which is independent of $N$ and $\ve$. Since $\tilde{F}'(s)=f(s)$ when $|s|\leq 2 q$, we infer that $u^{\ve, N}$  solves the Cauchy-Dirichlet problem  with $\tilde{F}'$ replaced by $f$. The rest of the proof for the existence of solutions $u^{\ve}$ to the Cauchy problem corresponding to \eqref{e:ap1} remains the same as for the case $F=0$, see \cite{BG1}. Since $F \in C^{2, \beta}_{loc}(\R)$, it follows from the Scahuder theory ( see Chapter 4  and Chapter 12 in \cite{Li}), that $u^{\ve} \in H_{3+\alpha}(\Rn \times [0, T])$ for some $\alpha > 0$ which depends on $\ve,p, n, q$ and $\beta$. We refer to Chapter 4 in \cite{Li} for relevant  notion of $H_{3+\alpha}$ spaces.

We note that the solutions $u^{\ve}$'s have  spatial gradient bounds, depending only on $n, p, q$ and $||Dg||_{L^{\infty}(\Rn)}$, which are uniform in $\ve$ for $\ve \leq 1$. This follows from Theorem 11.3 b)  in \cite{Li}. For this, one needs to observe that the limit behavior in  (11.17) in \cite{Li} is uniform in $\ve$, similarly to the case $F=0$. Now, as in the case $F=0$,  the uniform   bounds on the time derivatives of $u^{\ve}$, which depend only on the $C^{2}$ norm of $g$, can be  obtained by differentiating the approximating equations \eqref{e:ap1} with respect to the time variable and by applying Theorem \ref{max mod} above. Therefore, in the same way as for the case $F=0$, one can assert the existence of $u$ to \eqref{main5}  in the class $H_T$ when $g$ is smooth and has bounded derivatives of all orders.

In the case when $g$ is  only globally Lipschitz, we  take $\ve_{k}$-mollifications of $g$ for a sequence $\ve_{k}\to 0$, and call them $g_{k}$. Then, $g_{k}$ has bounded  derivatives of all order and
\begin{equation}\label{M101}
||g_{k}||_{L^{\infty}(\Rn)} \leq ||g||_{L^{\infty}(\Rn)},\ \ \ ||Dg_{k}||_{L^{\infty}(\Rn)} \leq ||Dg||_{L^{\infty}(\Rn)}.
\end{equation}
Let $u_{k}$ be the solution to the Cauchy problem corresponding to the initial datum $g_{k}$. As mentioned above, thanks to Theorem 11.3 in \cite{Li} ensures that $||Du_{k}||_ {L^{\infty}(\Rn \times (0, T))}$ is bounded  uniformly in $k$ by constants which depends only on $||Dg||_{L^{\infty}(\Rn)}$, $q$, $p$ and $n$. Since $g_{k} \to g$ uniformly in $\Rn$,  by the maximum modulus principle Theorem \ref{max mod} above we conclude that $u_{k} \to u$ uniformly in $\Rn \times [0, T]$, where $u$ is the unique solution to the Cauchy problem \eqref{main5} in the class $H_T$  corresponding to the  initial datum $g$.

\begin{rmrk}\label{assumpt1}
We note  that  the assumption \eqref{assump} is only used to assert a bound on $u^{\ve, N}$ independent of $\ve$ and $N$ as an intermediate step. If we  instead assume that  $f$ is bounded,  it turns out that  $w= ||g||_{L^{\infty}(\Rn)} + M_1t$  is a supersolution to the equation satisfied by $u^{\ve, N}$ when  $\ve \leq 1$  and   $M_1$ is chosen large enough depending only on  $||f||_{L^{\infty}(\Rn)}$. Hence, such $w$   can be used as a barrier for  $u^{\ve, N}$ from above and one can similarly bound $u^{\ve, N}$ from below by using $-w$ which is a subsolution to the same equation.
\end{rmrk}
\begin{rmrk}\label{assumpt2}
When $1<p < 2$, the assumption \eqref{assump}  is not needed. In that case, let $\tilde F$ be a $C^{2, \beta}$ compactly  supported function such that $\tilde F(s)=F(s)$ when $|s| \leq ||g||_{L^{\infty}(\Rn)} +  2$. Then, for each $\ve>0$,  we solve the corresponding  Cauchy-Dirichlet problem as before    in $\Om^N_T$ with $\tilde F'$ instead of $f$ and denote the corresponding solutions by $u^{\ve, N}$. For all $\ve$ small enough depending only on $p$, $f$, $T$ and $||g||_{L^{\infty}(\Rn)}$, it turns out that  $w= ||g||_{L^{\infty}(\Rn)} + \frac{t}{T}$  is a  supersolution to the equation satisfied by $u^{\ve, N}$ and hence can be used to assert boundedness of $u^{\ve, N}$ from above. Similarly, the  subsolution $-w$  can be used to assert  boundedness for $u^{\ve, N}$ from below. Therefore, for all such small enough $\ve$, it follows from the definition of $\tilde{F}$ that $u^{\ve, N}$ solves the Cauchy-Dirichlet problem with $\tilde F'$ replaced with $f$. The rest of the proof remains the same. This procedure does not work in the case $p=2$. This is because  when the  approximating equation \eqref{goodapp} is computed  for $w$, the term $\ve^{2-p}\tilde{F}'(w)$  does not go to zero as $\ve \to  0$  in the case  $p =2$ and therefore one cannot assert that $w$ is a supersolution to \eqref{goodapp}. Therefore, one interesting aspect is that for $1<p <2$, one has existence of solution to the Cauchy problem \eqref{main5} without  any growth  assumption on $f$ due to the special structure of the equation unlike  what one needs in the general theory of  uniformly parabolic equations, see for instance Theorem 12.16 in \cite{Li}.
\end{rmrk}

\section{Proof of the main results}\label{proofs}

We first  prove an intermediate  crucial result which  asserts  gradient estimates  for solutions to the  approximating Cauchy problems \eqref{e:ap1}.   For each $\ve>0$, we define
\begin{equation}\label{P:1}
P_{\ve}(u^\ve)(x, s) = \xi_{\ve}( |Du^\ve(x, s)|^{2}) - 2 F(u^\ve(x,s)),
\end{equation}
where  $u^{\ve}$ is  a solution to \eqref{e:ap1}, and we have let 
\begin{equation}\label{phixi}
\xi_{\ve} (s)= 2s\phi_{\ve}' - \phi_{\ve},\ \ \ \ \text{with}\ \ \phi_\ve(s)= \frac{2}{p}(\ve^{2} + s)^{p/2}.
\end{equation}

\begin{thrm}\label{aux}
Let $u^\ve$ be a solution of the approximating  equation \eqref{e:ap1} such that $u^\ve \in H_{3+\alpha}(\Rn \times [0, T])$ for some $\alpha > 0$.  If $P_\ve(u^\ve)(\cdot, 0) \leq 0 $, then  $P_\ve(u^\ve(x,t)) \leq 0$  for all $x\in \Rn$ and all $t\ge 0$.
\end{thrm}

\begin{rmrk}
Note that, when the initial datum $g$ has bounded derivatives of sufficiently high order (up to  order five), then the solutions $u^{\ve}$ constructed in Section \ref{existence1} satisfy the hypothesis of Theorem \ref{aux}.
\end{rmrk}

\begin{proof}[Proof of Theorem \ref{aux}]
Henceforth, we will routinely omit $\ve$-subscripts and superscripts, and suppress the dependence of $P$ on $u$. Thus, for instance, we will write $u$ instead of $u^\ve$, $P$ instead of $P_\ve(u^\ve)$. We will also write $\phi$ and $\xi$, instead of $\phi_\ve$ and $\xi_\ve$ like in \eqref{phixi}.  Note that the approximating equation can be rewritten as 
\begin{equation}\label{approx}
\operatorname{div} (\phi'(|Du|^2)Du)= f(u) + \phi'(|Du|^{2})u_{t}.
\end{equation} 
We let $\Lambda = \xi'$, and note that for  each $\ve > 0 $ we have from \eqref{phixi}
\begin{equation}\label{e:lambda}
\Lambda  = ( \ve^2 + |Du|^{2})^{p/2 - 2}( \ve^2 + (p-1)|Du|^{2}) > 0.
\end{equation}
We next write  \eqref{approx} in the following manner
\begin{equation*}
a_{ij} (Du)\ u_{ij}= f(u) + \phi'\  u_{t},
\end{equation*}
where   
\begin{equation}\label{nondiv}
a_{ij}= 2\phi'' u_{i}\ u_{j} + \phi' \delta_{ij}.
\end{equation}
Therefore, $u$ satisfies
\begin{equation}\label{e:650}
d_{ij}\  u_{ij} = \frac{f}{\Lambda} + \frac{\phi'}{\Lambda} u_{t},
\end{equation}
where $d_{ij}= \frac{a_{ij}}{\Lambda}$.
By differentiating \eqref{nondiv} with respect to $x_{k}$, we obtain
\begin{equation}\label{f:135}
(a_{ij}\ (u_k)_{i})_{j}= f'\ u_k  + \phi' \ u_{tk} + 2 \phi''\ u_{hk}\ u_{h}\ u_{t}.
\end{equation}
From the definition of $P$  in \eqref{P:1} we have,
\begin{equation}\label{P1}
P_{i} = 2 \Lambda u_{ki}\  u_{k} - 2 f\ u_{i},\ \ \ P_{t} = 2 \Lambda u_{kt}\ u_{k} - 2f\  u_{t}.
\end{equation}

We  now consider the following auxiliary function
\begin{equation*}
w = w_R = P -  \frac{M}{R} \sqrt{|x|^2 + 1} -   \frac{ct}{R^{1/2}},
\end{equation*}
where $R > 1$ and  $M$, $c$ are to be  determined subsequently. 
Note that $P \ge w$ for $t\ge 0$.
Consider the cylinder  $ Q_{R}=  B( 0 , R) \times  [0, T]$. One can see that  if $M$  is   chosen large enough,  depending on the $L^{\infty}$ norm of $u$ and its first derivatives, then $ w < 0 $ on the lateral boundary of $Q_{R}$. In this situation we see that if $w$ has a strictly positive maximum at a point $(x_0, t_0)$,
then such point cannot be on the parabolic boundary of $Q_R$. In fact, since $w<0$ on the lateral boundary, the point cannot be on such set. But it cannot be on the bottom of the cylinder either since at $t=0$ we have  $ w(\cdot,0) \leq P(u(\cdot,0)) = P(g) \leq 0$, where in the last inequality we have used the hypothesis.

Our objective is to prove the following claim: 
\begin{equation}\label{claim}
w \le K \overset{def}{=} R^{- \frac p2},\ \ \ \ \text{in}\ Q_R,
\end{equation}
provided that $M$ and $c$ are chosen appropriately.
This claim will be established in \eqref{claim2} below. 
We  first  fix  a point  $(y, s)$ in $\Rn$. Now for all  $R$ sufficiently  large enough, we  have that  $(y, s) \in Q_{R}$. We would like to emphasize over here that  finally we  let $R \to \infty$.  Therefore, once \eqref{claim} is established, we obtain from it  and the definition of $w$ that
\begin{equation}\label{e:68}
P(u)(y, s) \leq  \frac{K'}{R^{1/2}},
\end{equation}
where $K'$ depends on $\ve, (y,s)$ and the bounds of the derivatives of $u$  of order three. By letting  $ R \to \infty$ in \eqref{e:68}, we find  that  
\begin{equation}\label{est}
P(u)(y, s) \leq 0. 
\end{equation}
The sought for conclusion thus follows  from the arbitrariness of the point $(y, s)$.

In order to prove the claim \eqref{claim} we argue by contradiction and suppose that there exist $(x_0,t_0)\in \overline Q_{R}$ at which $w$ attains it maximum and for which 
\begin{equation*}
w(x_0, t_0) > K.
\end{equation*}
This implies that $(x_{0}, t_{0})$ is not on the parabolic boundary of $Q_{R}$. Note that from the definition \eqref{P:1} of $P$, we have
\[
\frac 12 P = (\ve^2 + |Du|^2 )^{\frac p2 - 1} \bigg[\frac{1}{p'}|Du|^2 - \frac{\ve^2}{p}\bigg] - F(u).
\]
Since $1<p\le 2$, we have $2\le p'<\infty$, and so $\frac{1}{p'} \le \frac 12<1$. Thus, at every point of $Q_R$ we have
\[
\frac 12 w \le \frac 12 P \le \frac{1}{p'}(\ve^2 + |Du|^2 )^{\frac p2 - 1} |Du|^2 < (\ve^2 + |Du|^2 )^{\frac p2 - 1} |Du|^2.
\]
It follows that at $(x_0,t_0)$ we must have
\begin{equation}\label{expl}
 (\ve^2 + |Du(x_0,t_0)|^2 )^{\frac p2 - 1} |Du(x_0, t_0)|^2\geq  \frac{1}{2} P (x_0, t_0) \geq \frac{1}{2} w (x_0, t_0)>  \frac{1}{2} K,
\end{equation}
which implies, in particular, that $Du(x_0, t_0) \neq 0$. Therefore, since $1 < p \leq 2$,  we obtain from \eqref{expl}
\begin{equation}\label{est5}
|Du(x_0,t_0)|^{p} \geq  ( \ve^2 + |Du(x_0,t_0)|^2 )^{p/2 - 1} |Du(x_0,t_0)|^2 \geq  \frac{1}{2} P (x_0, t_0) >  \frac{1}{2} K.
\end{equation}
On the other hand, since $(x_0,t_0)$ does not belong to the parabolic boundary, from the hypothesis that $w$ has its maximum at such point, we conclude that $w_t(x_0,t_0) \ge 0$ and $Dw(x_0,t_0) = 0$. These conditions translate into
\begin{equation}\label{M1}
P_{t} \geq  \frac{c}{R^{1/2}},
\end{equation} 
and
\begin{equation}\label{M2}
P_{i} =  \frac{M} {R} \frac{x_{0,i}}{ (|x_{0}|^2 + 1 )^{1/2}}.
\end{equation}
Now
\begin{equation*}
(d_{ij}w_{i})_{j} =   (d_{ij} P_{i})_{j} - \frac{M}{R} (d_{ij}  \frac{x_{i}}{ (|x|^2 + 1 )^{1/2}})_{j},
\end{equation*}
where
\begin{equation}\label{calc1}
(d_{ij}P_{i})_{j} = 2 ( \frac{ a_{ij}}{\Lambda} ( \Lambda u_{ki}\ u_{k} - f\ u_{i}))_{j} = 2( a_{ij}\ (u_{k})_{i}\  u_{k})_{j} - 2 (f\ d_{ij}\ u_{i})_{j}.
\end{equation}
After a simplification, \eqref{calc1} equals
\begin{equation*}
2 a_{ij}\  (u_{ki})_{j}\ u_{k} + 2 a_{ij}\ u_{ki}\ u_{kj} - 2 f'\ d_{ij}\  u_{i}\ u_{j} - 2 f\  d_{ij}\ u_{ij} - 2f\ (d_{ij})_{j}\ u_{i}.
\end{equation*}
We notice that
\begin{equation*}
d_{ij}u_{i}u_{j} = \frac{ 2 \phi^{''} u_{i}\ u_{j}\ u_{i}\ u_{j}  + \phi'\ \delta_{ij}\ u_{i}\ u_{j}}{\Lambda} = |Du|^2.
\end{equation*}
Now by using \eqref{f:135} and by  cancelling  the  term  $2f'|Du|^2$, we get that the right-hand side in \eqref{calc1} equals
\begin{equation*}
{\ds 2 \phi' u_{tk}\ u_{k}  + 4 \phi^{''}\ u_{hk}\ u_{h}\ u_{k}u_{t} + 2 a_{ij}\ u_{ki}\ u_{kj} - 2 f d_{ij}\ u_{ij} - 2 f d_{ij, j}\  u_{i}}.
\end{equation*}
 Therefore by using the equation \eqref{e:650},  we obtain 
\begin{align}\label{e:100}
(d_{ij}P_{i})_{j} & = 2 a_{ij}\ u_{ki}\ u_{kj}  +  2\phi^{'}\ u_{tk}\ u_{k} + 4 \phi^{''}\ u_{hk}\ u_{h}\ u_{k}\ u_{t}
\\   
& -  2 \frac{f^2}{\Lambda}  - 2 \frac{ f\ \phi^{'}\ u_t}{\Lambda} - 2fd_{ij,j}\ u_{i}.
\notag
\end{align}
By using the extrema conditions \eqref{M1}, \eqref{M2}, we have the following two conditions at $(x_0, t_0)$ 
\begin{equation}\label{M3}
u_{kh}\ u_{k}\ u_{h} = \frac{f}{\Lambda} |Du|^{2} +  \frac{M}{2 R \Lambda} \frac{x_{h}\  u_{h}}{(|x|^2 + 1)^{1/2}},
\end{equation}
\begin{equation}\label{M4}
2 \Lambda\  u_{kt}\ u_{k}  \geq   2 f u_{t}  + \frac{c}{R^{1/2}}.
\end{equation}
Using the extrema conditions   and by canceling $ 2 \phi^{'} u_{tk}u_{k} $ we  obtain,
\begin{align}\label{11}
(d_{ij} w_{i})_{j}   \geq   & 2a_{ij}\ u_{ki}\ u_{kj}  +   \frac{4 \phi^{''}\ f}{\Lambda} |Du|^{2} u_{t}  -  \frac{2 f^{2}}{\Lambda} - 2 f d_{ij, j}\ u_{i}
\\
 + &  \frac{2 \phi^{''}\ M\ x_h\ u_h\ u_t}{R\ \Lambda\ (|x|^2 + 1)^{1/2}} + \frac{c\ \phi'}{R^{1/2} \Lambda}  -  \frac{M}{R} ( d_{ij} \frac{x_{i}}{ (|x|^2 + 1)^{1/2}})_{j}.
 \notag
 \end{align}
Now  we  have the following  structure equation, whose proof is lengthy but straightforward,
\begin{equation}\label{e:100}
d_{ij, j } u_i = \frac{2 \phi^{''}} {\Lambda} ( |Du|^{2} \Delta u - u_{hk}\ u_h\  u_k).
\end{equation}
Using \eqref{M4} in \eqref{e:100}, we  find
\begin{equation*}
d_{ij, i}\ u_i= \frac{2 \phi^{''} |Du|^2}{\Lambda} (  \Delta u -   \frac{f}{\Lambda} -  \frac{M\ x_h\  u_h}{ 2 R\ |Du|^2\   \Lambda (|x|^2 + 1 )^{1/2}}).
\end{equation*}
Using the equation \eqref{approx},  we have 
\begin{equation*}
2 \phi^{''}\ u_{hk}\ u_h\  u_k + \phi'\  \Delta u  = f + \phi'\  u_t.
\end{equation*}
 Therefore,
\begin{equation}\label{M100}
\Delta\  u = \frac{ f + \phi'\  u_t - 2 \phi^{''}\  u_{hk}\ u_h\  u_k}{\phi'}.
\end{equation}
Substituting the value for $\Delta u$ in \eqref{M100}  and by using the extrema condition \eqref{M4}, we have the following equality  at $(x_0, t_0)$,
\begin{align}\label{e:60}
d_{ij, j}\ u_i  &  = \frac{2 \phi^{''}\ |Du|^2} {\Lambda\ \phi' }\bigg[ f + u_t\ \phi'  - 2 \phi^{''}\ \frac{|Du|^2}{\Lambda} f  - f  \frac{\phi'}{\Lambda}
\\ 
 &  - \frac{  \phi^{''}\ M\ x_h\  u_h } { R \Lambda (|x|^2 + 1 )^{1/2}}-  \frac{M\ x_h\  u_h\ \phi'}{  2 R\ |Du|^2\  \Lambda\  (|x|^2 + 1 )^{1/2}}\bigg].
 \notag
 \end{align}
Using the definition of $\Lambda$  and cancelling terms in \eqref{e:60}, we have that  the right-hand side  in \eqref{e:60}  equals
\begin{equation}\label{e:61}
2 \phi^{''}\frac{ |Du|^2 u_t }{\Lambda} -   \frac{ \phi^{''} M\ x_h\ u_h }{ \Lambda^2\  R\ (|x|^ 2 + 1 )^{1/2}}  -    \frac{ 2 (\phi^{''})^2\ |Du|^2\ M\ x_h\  u_h } {  R\ \Lambda^2\ \phi^{'}\  (|x|^2 + 1 )^{1/2}}. 
\end{equation}
Therefore, by  canceling the terms $4 \phi^{''} f \frac{|Du|^ 2 u_t }{ \Lambda}$ in \eqref{11}, we  obtain the following differential inequality at $(x_0, t_0)$,
\begin{align}\label{e:62}
(d_{ij}w_{i})_{j} \ge &  \frac{c\ \phi'} {  R^{1/2}\ \Lambda} -   \frac{2\ f^2}{\Lambda}   - \frac{M}{R}\  ( d_{ij}\ \frac{x_{i}}{ (|x|^2 + 1)^{1/2}})_{j}  +    \frac{2 \phi^{''}\ M\ x_h\ u_h\ u_t}{R\ \Lambda (|x|^2 + 1)^{1/2}} 
\\  
& + \frac{2 f\ \phi^{''} M\ x_h\ u_h }{ \Lambda^2\ R\ (|x|^ 2 + 1 )^{1/2}} +   \frac{4 f\ (\phi^{''})^2\ |Du|^2\ M\ x_h\  u_h } {R\ \Lambda^2\ \phi^{'}\ (|x|^2 + 1 )^{1/2}} + 2 a_{ij}\ u_{ki}\ u_{kj}.
\notag
\end{align}

Now by using the identity  for $ DP$ in \eqref{P1} above, we have
\begin{equation}\label{e:63}
{\ds u_{ki}\ u_{kj}\ u_{i}\ u_{j}= \frac{( P_k + 2 fu_k ) ^{2}} { 4  \Lambda ^ {2} }}.
\end{equation}
Also,
\begin{equation*}
a_{ij}\  u_{kj}\  u_{ki}  = \phi'\  u_{ik}\ u_{ik}  + 2 \phi^{''}\ u_{ik}\ u_{i}\ u_{jk}\ u_{j}.
\end{equation*}
Therefore, by Schwarz  inequality, we have
\begin{equation*}
a_{ij}\ u_{kj}\ u_{ki} \geq  \phi'  \frac{ u_{ik}\ u_{jk}\  u_i\  u_ j}{  |Du|^2 } + 2 \phi^{''}\  u_{ik}\ u_i\  u_{jk}\ u_{j}   = \frac{\Lambda u_{ik}\ u_{i}\ u_{jk}\ u_{j}}{|Du|^2}.
\end{equation*}
Then, by using \eqref{e:63}  we find
\begin{equation}\label{e:64}
a_{ij}\ u_{kj}\ u_{ki} \geq \frac{(P_{k} + 2fu_k)^{2}}{4 \Lambda |Du|^{2}} =  \frac{  |DP|^2 + 4 f^2 |Du|^2 + 2 f < Du,DP> }{  4 |Du|^ 2 \Lambda }.
\end{equation}
At this point, using \eqref{e:64} in \eqref{e:62}, we can cancel off $\frac{2f^2}{\Lambda}$ and  consequently obtain the following inequality at  $(x_0, t_0)$,
\begin{align}\label{e:65}
(d_{ij}w_{i})_{j} \geq & \frac{c \phi'} {  R^{1/2} \Lambda}   +   \frac{ f < Du, DP> }{  |Du|^2 \Lambda }   - \frac{M}{R}\ ( d_{ij}\ \frac{x_{i}}{ (|x|^2 + 1)^{1/2}})_{j} + \frac{2\ \phi^{''}\ M\ x_h\ u_h\  u_t}{R\ \Lambda (|x|^2 + 1)^{1/2}}  
\\
& + \frac{4 f\ (\phi^{''})^{2}\ |Du|^{2} M\ x_h\  u_h } {   R\ \Lambda^{2}\ \phi' (|x|^2 + 1 )^{1/2}}   +  \frac{2 f\ \phi^{''} M\ x_h u_h }{ \Lambda^2\ R\ (|x|^ 2 + 1 )^{1/2}}.
\notag
\end{align}
By assumption, since  $ w(x_0, t_0) \geq K$, we have that 
\begin{equation*}
 |D u |  \geq \frac{1}{2^{1/p}R^{1/2}}.
\end{equation*}
Moreover, since  $u$ has  bounded derivatives of  upto order 3, for a fixed $\ve > 0$, we  have that    $\phi' $ and  $\Lambda$ are  bounded from below by a positive constant. Therefore by \eqref{M2},  the term $ \frac{ f < Du, DP> }{  |Du|^2 \Lambda }$ can be controlled from below by   $ - \frac{ M^{''}}{ R^{1/2}}$ where $M^{''}$ depends on $\ve$ and  the bounds of the derivatives of $u$. Consequently, from \eqref{e:65}, we have at $(x_0, t_0)$, 
\begin{equation}
(d_{ij}w_{i})_{j} \geq     \frac{C(c)} { R^{1/2}} -   \frac{L(M)} {R} -  \frac{ M''}{ R^{1/2}}.
\end{equation}
Now in the very first place, if  $c$ is chosen large enough  depending only on $\ve$ and  the bounds of the derivatives of $u$ up to order three,  we would have the following inequality $ \text{at}\ (x_0, t_0)$, 
\begin{equation*}
(d_{ij}w_{i})_{j}   > 0.
\end{equation*}
This contradicts the fact that $w$ has a maximum at  $(x_{0}, t_{0})$.
Therefore, either  $ w (x_{0}, t_{0}) < K$, or    the maximum of $w$  is achieved on the parabolic boundary where $w < 0$.  In either case, for an arbitrary point $(y,s)$ such that $|y| \leq R$,  we  have that 
\begin{equation}\label{claim2}
w(y, s) \leq \frac{1} {R^{p/2}}.
\end{equation}

\end{proof}

\begin{proof}[Proof of Theorem \ref{1d}]
Let  $g_{k}$ be the $\ve_{k}$  mollifications of $g$ which converges to $g$ uniformly in $\Rn$ as $k \to \infty$. Note that $g_{k}$ has bounded derivatives of all orders with bounds depending on $\ve_{k}$. Given any $\delta >0$, we  note that for large enough $k$, $g_{k}$ satisfies  \eqref{e:main1} with $F$ replaced by $G= F + \delta$. This can be seen as follows:
\begin{equation}\label{h3}
|Dg_{k}(x)|= | \int_{\Rn} Dg (x-y) \rho_{\ve_k} (y) dy | \leq \int_{\Rn} |Dg(x-y)| \rho_{\ve_k} (y) dy
\end{equation}
See for instance Theorem 6.25 in \cite{R}. We choose to cite  this reference since the integrals considered in \eqref{h3} are vector valued and we need to make sure that no additional constants are incurred in front of the last integral in \eqref{h3}. Therefore,
\begin{equation}\label{h}
|Dg_{k}(x)|^{p} = | \int_{\Rn} Dg (x-y) \rho_{\ve_k} (y) dy |^{p} \leq  \int_{\Rn} |Dg(x-y)|^{p} \rho_{\ve_k} (y) dy
\end{equation}
The last inequality in \eqref{h} follows from \eqref{h3} and  Jensen inequality. Now  since $|Dg|^{p} \leq \frac{p}{p-1} F(g)$ a.e., we have  for all $k$ large enough, 
\begin{equation}\label{h2}
|Dg_{k}(x)|^{p} \leq \frac{p}{p-1} \int_{\Rn} F( g(x-y)) \rho_{\ve_k} (y) dy \leq \frac{p}{p-1}\underset{B_{\ve_k}(x)}{\sup}\ F(g) \leq\frac{p}{p-1}( F(g_k(x)) + \delta).
\end{equation}
In the last inequality in \eqref{h2}, we have made use of the fact that $g_{k}$ converges to $g$ uniformly in $\Rn$ since $g$ is globally Lipschitz. This justifies the claim above.

Now for each such $k$, let $u^{\ve}_{k}$ be the solution to the Cauchy problem corresponding to equation \eqref{e:ap1} with initial datum $g_{k}$. We furthermore assume that for $1 < p < 2$, $\ve$ is small enough so that the conditions in Remark \ref{assumpt2} is satisfied.
We now  note that for $1 < p \leq 2$, 
\begin{equation}
 \tilde{P_{\ve}} \leq  \tilde{P},
\end{equation}
where $\tilde{P}, \tilde{P_{\ve}}$  are  defined  as in \eqref{P:1}. Therefore,  since  $\tilde{P}(g_{k}) \leq 0$,  we have that $\tilde{P_{\ve}}(g_{k}) \leq 0$.  Theorem \ref{aux} applied to $u^{\ve}_{k}$ implies that $\tilde{P_{\ve}}(u^{\ve}_{k}) \leq 0$ for all positive times.
Now, by Dini's theorem the functions $h_{\ve}( x) = (\ve ^ 2 + |x|^2)^{p/2 - 1} |x|^{2}$  $\to$  $|x|^{p}$, uniformly on compact sets. Thus, because of  uniform bounds on the gradients, given any $\gamma > 0 $ for all small enough $\ve$ we have that at each time level  $t$,
\begin{equation}\label{f:50}
\frac{p-1}{p} |Du^{\ve}_{k}|^{p}  \leq  G(u^{\ve}_{k})   + \gamma.
\end{equation}
Integrating \eqref{f:50} over an open ball  $B_{r}= B_{r}(x)$  where $x$ is any arbitrary point, by using lower semicontinuity on the left-hand  side, and  by  passing to the limit in $\ve$ on the right-hand side, and then by letting $\gamma  \to 0$, we find 
\begin{equation}\label{f:51}
\frac{p-1}{p} \int_{B_r} |Du_{k}|^{p}  \leq \int_{B_{r}} G(u_{k}). 
\end{equation}
Now by the maximum modulus principle, Theorem \ref{max mod}, $u_{k} \to u$ uniformly in $\Rn$ and weakly in $W^{1,p}_{loc}(\Rn)$ at any given time $t$ where $u$ is the solution to the Cauchy problem with initial datum $g$. Therefore in \eqref{f:51},   by using lower semicontinuity on the left hand side and by passing to the limit in $k$ on the right hand side, we have that \eqref{f:51} holds for $u$.  Then from the Lebesgue differentiation theorem, it follows that  at  a given time level $t$, 
\begin{equation}\label{e:210}
\frac{p-1}{p} |Du|^{p}  \leq  F(u)   + \delta \ \ \ \ \text{a.e. in}\ \Rn.
\end{equation}
By letting $\delta \to 0$, we reach the desired  conclusion.

\end{proof}

\begin{proof}[Proof of Theorem \ref{2d}]
Since by hypothesis, $g$ has bounded derivatives of upto order $2$, we  have by an application of maximum  principle   as described in Section \ref{existence1}, that  the solutions $u^{\ve}$ to the approximating equations \eqref{e:ap1} are such that $|Du^{\ve}|$ and $|u^{\ve}_{t}|$ are  bounded from above by  constants  which are independent of $\ve$.  Therefore, in the case $n=2 $, we see that, because of uniform bounds on the space and time  derivatives, at each time level $t$ the solutions $u^{\ve}$ to the approximating equations  \eqref{e:ap1} solve a  uniformly elliptic linear PDE in non-divergence form with right-hand side uniformly bounded in $\ve$. Therefore, from Theorem 12.4 in \cite{GT} (see also \cite{T}),  it follows   that $Du^{\ve}(\cdot, t)$  has uniform H\"{o}lder bounds independent of $\ve$ with an exponent $\alpha$ which only depends on $p$. Consequently,   $Du(\cdot, t)$  is H\"{o}lder continuous in $x$ and the conclusion follows. 
 
\end{proof}

\noindent  Now we turn our attention to the proof of Theorem \ref{3d} which  is similar to that for the elliptic case in \cite{CGS}.
\begin{proof}[Proof of Theorem \ref{3d}]
Via an approximation argument as used before in the proof of Theorem \ref{1d}, we can   assume that the the initial datum $g$  has bounded  derivatives of  sufficiently high order. Let $u^{\ve}$ be the solution to \eqref{e:ap1} corresponding to initial datum $g$.
We consider   the function 
\begin{equation*}
\psi_{\ve}(s) = u^{\ve}(x_{1} + s \omega,t_0) - u^{\ve}(x_{1},t_0)
\end{equation*}
for  some $x_{1} \in \Rn$  where $\omega$ is some unit direction. The point $x_1$ is going to be chosen appropriately later. From the definition, we have that $\psi_\ve(0) = 0$ and
\begin{equation*}
 |\psi'_{\ve}(s)| \leq |Du^{\ve}( x_{1} + s\omega,t_0)|.
\end{equation*}
We  now  define the function 
\begin{equation*}
\xi^{\ve} (s) = 2 s \phi'_{\ve} - \phi_{\ve}  + \frac{2}{p} \ve^2.
\end{equation*}
For  $\delta$ small  enough, let
\begin{equation*}
 G_{\ve}=  \xi^{\ve} -  \delta (\ve^2 + s)^{p/2}. 
\end{equation*}
Clearly, $G_{\ve}( 0 ) =  - \delta \ve^p$ and by the ellipticity it is easily seen that $G'_{\ve} \geq 0$. This implies that
\begin{equation}\label{just}
G_{\ve}(s) \geq - \delta \ve^p
\end{equation}
Therefore from \eqref{just} and the definition of $G_{\ve}$, given any $\gamma >0$, for small enough $\ve$, 
\begin{equation*}
|Du^{\ve}(x_{1}  + s\omega,t_0)|^{p} \leq  C  \xi^{\ve} ( |Du^{\ve}(x_{1} + s\omega,t_0)|^2)   + \delta \ve^{p}   + \gamma.
\end{equation*}
By applying Theorem \ref{aux}, we thus obtain
\begin{equation}\label{e:140}
|\psi'_{\ve}(s)|^{p} \leq C ( F (  u^{\ve}(x_{1} + s \omega,t_0))   + k(\ve)  + \gamma,
\end{equation}
where $k(\ve) \to  0$ as $\ve \to 0$. Repeating the arguments in the proof of Theorem \ref{1d}, and finally letting $\gamma \to 0$, we obtain
\begin{equation}\label{e:141}
|\psi' (s)|^{p} \leq  C  F(u(x_{1}+ s\omega, t_0)) \ \ \  \text{a.e. in  s }
\end{equation}
where 
\begin{equation*}
\psi (s) = u(x_{1} + s\omega,t_0) ) - u(x_{1},t_0).
\end{equation*}
Now suppose that  $F(u(x_{0}, t_{0}))= 0 $, and let  $u_{0}= u(x_{0}, t_{0})$. Indicating with $\Pi_{x}$ the projection onto the $x$-component, consider the set  $V=  \Pi_{x} (u^{-1}( u_{0}) \cap \Rn \times \{ t_{0}\})$, and let $x_1 \in V$. Clearly, $V$ is closed.  Since  $F \geq 0 $ and  $F(u_{0}))= F(u(x_{1}, t_0))= 0 $, we have that 
\begin{equation}
F(u-u_0)= O((u-u_0)^2)
\end{equation}
 Hence for $s$ small enough,
\begin{equation*}
 F(u(x_{1} + s\omega, t_0)) \leq  K |\psi(s)|^2.
\end{equation*}
Therefore from \eqref{e:141}, we have for all such $s$  in a small enough interval which does not depend on $\omega$,
\begin{equation*}
|\psi'(s)| \leq C|\psi(s)| \\ \\\ \text{a.e.}
\end{equation*}
This implies  $\psi= 0$ in  that same interval. Since $\omega$ is arbitrary, this implies that $V$ is open   and hence equals the whole of  $\Rn$. The desired conclusion  thus follows.

\end{proof}

\begin{rmrk}
We  would like the reader to note that the reason  for which we  employ the regularization scheme $ u^{\ve}$'s which are solutions to \eqref{e:ap1}  in the proof of Theorem \ref{3d} as an intermediate step  is because  we can only assert that  the corresponding gradient estimate \eqref{e:main} for $u$   holds a.e  in $\Rn$.  Therefore, it need not hold on the $1$ dimensional  line $[x+ s \omega: s \in \R]$.
\end{rmrk}

\end{document}